\documentclass[conference]{IEEEtran}
\IEEEoverridecommandlockouts
\usepackage{cite}
\usepackage{amsmath,amssymb,amsfonts,url,bm,enumerate}
\usepackage{amsthm}
\usepackage{graphicx}
\usepackage{textcomp}
\usepackage[dvipsnames]{xcolor}
\usepackage{tikz}
\usetikzlibrary{arrows.meta}
\usepackage{float}
\usepackage{colortbl}
\usepackage{booktabs}
\usepackage{array}
\usepackage{comment}
\usepackage{multicol}
\usepackage{enumitem}
\def\BibTeX{{\rm B\kern-.05em{\sc i\kern-.025em b}\kern-.08em
    T\kern-.1667em\lower.7ex\hbox{E}\kern-.125emX}}
    
\usepackage{hyperref}%%JP
\usepackage{pict2e}%%KE

\usepackage{tikz-cd}
\tikzset{commutative diagrams/.cd}
\tikzcdset{row sep/tiny=0.12 em}

\usetikzlibrary{intersections}

\newlength\myheight

        \usetikzlibrary{calc}

\theoremstyle{plain}
\newtheorem{theorem}{Theorem}%[section]
\newtheorem{lemma}[theorem]{Lemma}
\newtheorem{proposition}[theorem]{Proposition}

\newtheorem{corollary}[theorem]{Corollary}

\theoremstyle{definition}
\newtheorem{definition}[theorem]{Definition}

\newtheorem{remark}[theorem]{Remark}

\newtheorem{example}[theorem]{Example}

\newtheorem{claim}{Claim}[theorem]
 
\newcommand{\setin}[3]{\{#1\in#2\;|\;#3\}}

\newcommand{\Naturals}{{\mathbb N}}

\newcommand{\Reals}{{\mathbb R}}

\newcommand{\auv}[1]{``#1''}

\newcommand{\downset}{\mathop{\downarrow}\!}
\newcommand{\upset}{\mathop{\uparrow}\!}

\newcommand{\conjun}{\mathrel{\wedge}}
\newcommand{\disjun}{\mathrel{\vee}}

\definecolor{darkgreen}{rgb}{0.0, 0.5, 0.0}

\newcommand{\lin}[1]{[#1]}

\newcommand{\tikzperpcircle}{%
    \tikz[baseline=(C.base)] {
        \draw[line width=0.4pt] (0,0) -- (0,0.4ex); % svislá čára
        \draw[line width=0.4pt] (-0.2ex,0) -- (0.2ex,0); % vodorovná čára
        
        \node[draw, circle, inner sep=0.5pt, minimum size=1.1ex] (C) at (0, 0.2ex) {};
    }%
}

\newcommand{\xprime}{^{\tikzperpcircle}}
\newcommand{\xxprime}{^{{\tikzperpcircle}{\tikzperpcircle}}}

\definecolor{brightgreen}{rgb}{0.4, 1.0, 0.0}
\newcolumntype{?}{!{\vrule width 1pt}}

\begin{document}

\title{Quantales carrying ortholattice structure\\
}
\author{
\IEEEauthorblockN{Michal~Botur}
\IEEEauthorblockA{\textit{Department of Algebra and Geometry} \\
	\textit{Faculty of Science, Palack\'y University Olomouc }\\
	Olomouc, Czech Republic\\
	michal.botur@upol.cz}
\and
%\\
\IEEEauthorblockN{David Kruml, Jan~Paseka}
\IEEEauthorblockA{\textit{Department of Mathematics and Statistics} \\
\textit{Faculty of Science, Masaryk University}\\
Brno, Czech Republic\\
kruml,paseka@math.muni.cz}
}

\maketitle

\begin{abstract}
This paper investigates the intersection of residuated structures from many-valued logic and orthomodular lattices from quantum logic. We explore whether non-Boolean structures can simultaneously satisfy residuation principles and orthocomplementation requirements. Our main contribution is a study of Girard posets with inversions, providing a characterization theorem where a unital residuated poset is Girard if and only if it admits an inversion satisfying specific adjointness conditions. We prove that any complemented lattice admitting an integral residuated structure must be Boolean, which motivates our search for orthomodular examples in the non-integral case. We answer this by demonstrating that the lattice $C(\mathbb{R}^n)$ of closed subspaces of $n$-dimensional real coordinate space carries both an orthomodular and a commutative Girard quantale structure. This construction provides a concrete non-Boolean framework unifying quantum-logical and many-valued logical reasoning.
\end{abstract}

\begin{IEEEkeywords}
Residuated lattice, quantale, Girard quantale, orthomodular lattice, inversion, linear negation, many-valued logic, quantum logic, cyclic dualizing element, Hilbert space
\end{IEEEkeywords}

\section{Introduction}
The landscape of non-classical logics has evolved along two distinct yet increasingly intertwined paths: many-valued logics, which emerged  from considerations of vagueness and partial truth, and quantum logics, which arise from the mathematical foundations of quantum mechanics. While Boolean algebras provide the algebraic semantics for classical propositional logic, the quest to understand non-classical reasoning has led to richer algebraic structures---residuated lattices for many-valued logics and orthomodular lattices for quantum logics.

{Multi-valued logic frameworks} traditionally employ residuated structures, where the residuation property captures the algebraic essence of logical implication. Residuated lattices, and their complete variants known as quantales, have proven fundamental in modeling fuzzy logic, intuitionistic logic, and particularly linear logic, where they offer precise semantics for resource-sensitive reasoning. Girard quantales, equipped with a cyclic dualizing element, emerge as the natural algebraic semantics for propositional linear logic, much as Boolean algebras serve classical logic.

{Quantum logic frameworks}, by contrast, center on orthomodular lattices---structures that capture the non-Boolean nature of quantum propositions while preserving the fundamental complementation structure inherent in quantum mechanics. The canonical example is the lattice of closed subspaces of a Hilbert space, where the orthocomplement corresponds to orthogonal complement and the non-distributivity of the lattice reflects the incompatibility of certain quantum observables.

This paper investigates the {intersection of these two algebraic traditions}, seeking structures that simultaneously respect the residuation principles of many-valued logic and the orthocomplementation of quantum logic. Our central motivation is to construct and study residuated orthomodular lattices --- a class of structures sitting at the crossroads of these frameworks.

We develop our analysis through the lens of {Girard posets with inversions}, extending previous work on Girard quantales to the broader setting of residuated posets. A key finding establishes that complemented lattices admit an integral residuated structure if and only if they are Boolean algebras (Theorem~\ref{thm:2}). This naturally raises the question: Do there exist orthomodular structures that are also unital commutative quantales but not Boolean algebras?

We answer this affirmatively by demonstrating that the lattice $C(\mathbb{R}^n)$ of closed subspaces of 
$n$-dimensional real coordinate space constitutes both an orthomodular lattice and a commutative Girard quantale (Theorem~\ref{Rnexample}). A key highlight of this construction is that the orthocomplement from the quantum-logical framework coincides precisely with the linear negation (algebraic inversion) of the Girard structure. This provides a concrete non-Boolean example that unifies quantum-logical and many-valued logical structures, offering new possibilities for modeling systems that exhibit both quantum uncertainty and resource-sensitive reasoning.

This paper assumes familiarity with the fundamental concepts of lattice theory and basic quantum logic. For inversions, 
ortholattices and orthomodular lattices and their properties, we refer readers to the monographs by Birkhoff~\cite{Bi} and Kalmbach~\cite{OML}. Background on quantales  can be found in the works of Rosenthal~\cite{rosenthal1}, Kruml and Paseka~\cite{KrPaseka08}. For fundamentals of the theory of operator algebras we refer to Kadison and Ringrose\cite{KaRi}.

The paper is organized as follows: Section~\ref{Preliminaries} establishes the foundational concepts of orthomodular lattices, residuated posets, and quantales. Section~\ref{Girposinv} develops the theory of Girard posets with inversions, characterizing when residuated structures admit cyclic dualizing elements. Section~\ref{quancarr} presents our main results on quantales carrying ortholattice structure, culminating in the construction of orthomodular Girard commutative quantales on real coordinate spaces. 
Our conclusions follow in Section~\ref{Conclusions}.

\section{Basic concepts}\label{Preliminaries}

\subsection{Inversions, ortholattices and orthomodular lattices}

\begin{definition}\label{def:inversion}
We say that an {\em inversion} $^{\xprime}$ is given on a poset $(P, \leq)$ if a map $(-)^{\xprime}: P \rightarrow P$ exists such that $x \le y \Leftrightarrow x^{\xprime} \ge y^{\xprime}$ and ${x^{\xxprime}} = x$ for every $x, y \in P.$
\end{definition}

Recall that an inversion is an involutive dual order automorphism.

\begin{definition}
\label{OMLatDef}
A meet semi-lattice $(X,\conjun, 1)$ is called an {\em ortholattice} if it
comes equipped with an inversion $(-)^{\perp}\colon X \to X$ satisfying:
\begin{itemize}
  % \item $x^{\perp\perp} = x$;
  % \item $x \leq y$ implies $y^\perp \leq x^\perp$;
   \item $x \conjun x^\perp = 1^\perp$.
\end{itemize}

\noindent One can then define a bottom element as $0 = 1 \conjun
1^{\perp} = 1^\perp$ and join by $x\disjun y = (x^{\perp}\conjun
y^{\perp})^{\perp}$, satisfying $x\disjun x^{\perp} = 1$.

We write $x\perp y$ if and only if $x\leq y^{\perp}$. 

Such an ortholattice is called {\em orthomodular} if it satisfies (one of)
the three equivalent conditions:
\begin{itemize}
\item $x \leq y$ implies $y = x \disjun (x^\perp \conjun y)$;

\item $x \leq y$ implies $x = y \conjun (y^\perp \disjun x)$;

\item $x \leq y$ and $x^{\perp} \conjun y = 0$ implies $x=y$.
\end{itemize}
\end{definition}

Our guiding example is the following.

\begin{example}\label{examHilb}
Let $H$ be a {Hilbert space} defined over the field of real numbers ($\mathbb{R}$) or complex numbers ($\mathbb{C}$). The closed subspace spanned by any subset $S \subseteq H$ is denoted by $\lin S$. We define the set of all closed subspaces as:
$$C(H) = \{ \lin S \mid S \subseteq H \}$$

The structure $C(H)$ is endowed with the operations of an {orthomodular lattice}, where:
\begin{enumerate}
    \item The meet operation $\wedge$ corresponds precisely to the set intersection $\cap$.
    \item The orthocomplement $P^{\perp}$ of any closed subspace $P \in C(H)$ is its {orthogonal complement} within $H$.
\end{enumerate}

Illustrative examples of such Hilbert spaces include:
\begin{itemize}
    \item The $n$-dimensional coordinate spaces $\mathbb{R}^n$ and $\mathbb{C}^n$.
    \item The vector spaces of $n \times n$ matrices, $M_n(\mathbb{R})$ and $M_n(\mathbb{C})$.
\end{itemize}
\end{example}

Let us recall the following elementary result.

\begin{lemma} {\rm\cite[Lemma 3.4]{Jac}}
\label{DownsetLem}
Let $X$ be an orthomodular lattice, with element $a\in X$. 
The (principal) downset $\downset a = \setin{u}{X}{u \leq a}$ is
  again an orthomodular lattice, with order, meets and
  joins as in $X$, but with its own orthocomplement $\perp_a$ given
  by $u^{\perp_a} = a \conjun u^{\perp}$, where $\perp$ is the
  orthocomplement from $X$.
\end{lemma}

\begin{definition}
   A bounded lattice $(X, \le, \wedge, \vee, 0,1)$
    is {\em complemented} if, for every element $x \in P$, there exists an element $x' \in P$ (the complement of $x$) such that:
    \begin{itemize}
        \item $x \vee x' = 1$,
        \item $x \wedge x' = 0$.
    \end{itemize}
\end{definition}

\begin{remark}\label{boolcompor}\hfill 
\begin{itemize}
    \item A Boolean algebra is a complemented lattice that is also distributive.
    \item An ortholattice is a complemented lattice where the complementation is unique and satisfies additional properties like involution ($x'' = x$) and order-reversal ($x \le y \Rightarrow y' \le x'$).
\end{itemize}
\end{remark}

\begin{definition}[Compatibility]
Two elements $x,y$ of an orthomodular lattice $X$ are \emph{compatible}, denoted $x \leftrightarrow y$, if
\[
x = (x \wedge y) \vee (x \wedge y^\perp).
\]
Equivalently, $a$ and $b$ generate a Boolean sublattice of $X$. A \emph{block} in an orthomodular lattice $X$ is a maximal Boolean sublattice of $X$, i.e., a Boolean sublattice not properly contained in any other Boolean sublattice of $X$.
\end{definition}

\subsection{Residuated posets}

\begin{definition}
\label{respos} A {\em residuated poset} is a structure 
$(P, \le, \odot, \rightarrow,$ $ \leftarrow)$ such that:
\begin{enumerate}
    \item $(P, \le)$ is a poset.
    \item $\odot$ is a binary operation on $P$ (the {\em multiplication}), which is associative.
    \item $\rightarrow$ and $\leftarrow$ are binary operations on $P$, called the {\em right residuum} and the {\em left residuum}, respectively.
    \item The triple $(\odot, \rightarrow, \leftarrow)$ satisfies the following {\em adjointness condition} for all $x, y, z \in P$:
    
        $$\begin{array}{c l}
        x \odot y \le z &\quad \text{if and only if} \quad x \le y \rightarrow z \\
        &\quad \text{if and only if} \quad y \le z \leftarrow x.
        \end{array}
        $$
  
\end{enumerate}
A residuated poset $(P, \le, \odot, \rightarrow, \leftarrow)$ is: 
\begin{enumerate}\item {\em commutative} if its semigroup is abelian,
\item {\em idempotent} if its semigroup is idempotent,
\item {\em unital} if its semigroup is a monoid with a unit $e$, 
\item {\em integral} if its semigroup is a monoid with a unit $1$, where 
$1$ is the greatest element of 
the poset $(P, \le)$.
\item {\em a residuated lattice} if $\leq$ is a lattice order.
\end{enumerate}
\end{definition}

\begin{remark}\label{remint}\hfill 

\begin{enumerate}
\item If the operation $\odot$ is commutative ($x \odot y = y \odot x$), then the right and left residua coincide ($\rightarrow$ is the same as $\leftarrow$), and the structure simplifies to a commutative residuated poset.
\item Multiplication preserves existing joins in both arguments and as right adjoints, $y \rightarrow (-)$ 
and $(-) \leftarrow x$ preserve existing meets.
\item For an integral residuated poset $P$ and for every $x, y\in P$ we obtain
$
x\odot y\le x, y \quad (x\odot y  \le x\odot 1 = x \text{ and similarly } x\odot y \le 1\odot y = y).
$
\item Every Boolean algebra is an integral commutative 
residuated lattice where $\odot=\wedge$ and the residuum operation is defined by a stipulation
$$
x\rightarrow y=x^{\perp}\vee y.
$$
\end{enumerate}
\end{remark}

An important subclass of residuated posets are Girard posets.

\begin{definition}[Girard posets] \cite{KrPaseka19} 
\label{def:Girpos}
An element $d$ of a residuated poset $(P, \le, \odot, \rightarrow, \leftarrow)$ is called \emph{dualizing} if for all $x \in P$,
\[
d\leftarrow (x \rightarrow d)  = x = (d\leftarrow x) \rightarrow d.
\]

An element $d \in P$ is \emph{cyclic} if for all $x,y \in P$,
\[
x\odot y \leq d\quad\text{if and only if}\quad  %
y\odot x \leq d.
\]
%\end{definition}
%\begin{definition}[Girard Quantale]
A \emph{Girard poset} is a residuated poset $(P, \le, \odot, \rightarrow, \leftarrow)$  that possesses a \emph{cyclic dualizing} element $d \in P$. 

In that case we write $x^{\xprime}=x\rightarrow d=d\leftarrow x$ and  ${}^{\xprime}$ is called a {\em linear negation} (algebraically, an inversion). From \cite[Proposition 2.2 and Proposition 2.3]{KrPaseka19}, we know that 
$P$ is a unital residuated poset with a unit $e=d^{\xprime}$, 
$x \rightarrow y=(x\odot y^{\xprime})^{\xprime}$ and 
$y \leftarrow x=(y^{\xprime}\odot x)^{\xprime}$.
\end{definition}

\subsection{Quantales}

While complete orthomodular lattices offer a static snapshot of a quantum system's possible states, primarily focused on testable properties, quantales provide a contrasting, dynamic viewpoint. They enable reasoning about the system's evolution and the structure of the quantum actions that drive these changes.

Mulvey \cite{Mulvey86} coined the term \auv{quantale} as a \auv{quantization} of the concept \auv{location} during an Oberwolfach Category Meeting in the early 1980s. Initially inspired by topology and functional analysis, quantales have developed into a robust foundation for modeling complex systems with non-classical logics across fields like computer science, abstract algebra, and quantum mechanics. A key breakthrough was the recognition that quantales provide the semantics for propositional linear logic, much like Boolean algebras do for classical propositional logic. Algebraically, quantales naturally arise as structures such as subgroups or lattices of ideals.

\begin{definition}\label{qsemiq}{\rm \cite{rosenthal1}
\begin{enumerate}
\item A {\em quantale\/} is a complete lattice $Q$ with an associative
binary multiplication satisfying
$$
x\odot\bigvee\limits_{i\in I}
x_i=\bigvee\limits_{i\in I}(x\odot
x_i)\ \ \hbox{and}\ \ \left(\bigvee\limits_{i\in I}x_i\right)\odot
x=\bigvee\limits_{i\in I}(x_i\odot x)
$$
for all $x,\,x_i\in Q,\,i\in I$ ($I$ is a set). Here 
$\bigvee\limits_{i\in I} x_i$ denotes the join of the set 
$\{x_i\colon {i\in I}\}$.
 The smallest element 
$0=\bigvee \emptyset$ of $Q$ is a {\em zero
element} of $Q$: $0 \odot s = 0 = s\odot 0$.

\item By an {\em involutive  quantale} will be meant
a  quantale $Q$ together with a semigroup
 involution $^{*}$ satisfying
$$
(\bigvee a_{i})^{*}=\bigvee a_{i}^{*}
$$
\noindent for all $a_{i}\in Q$. 
\end{enumerate}}
\end{definition}

Note that every quantale $Q$ is a residuated complete lattice such that 
$$
\begin{array}{r c l}
y \rightarrow z&=&%
\bigvee \{t \in Q\mid  t \odot y \le z\}\quad  \text{and}\\ 
z \leftarrow x&=&\bigvee \{t \in Q\mid  x \odot t \le z\}.
\end{array}
$$

Hence we can easily transfer notions from residuated posets 
to quantales. In particular, a {\em Girard quantale} is 
a complete Girard poset. Any complete Boolean algebra is a Girard quantale with dualizing element $0$.

Girard quantales play a significant role in the algebraic semantics of linear logic, especially in modeling the multiplicative and exponential fragments. 

\section{Girard posets and inversions}\label{Girposinv}

Inspired by the study in \cite{smarda97} on Girard quantales and inversion operations, this work analyzes the algebraic properties of Girard posets, with a particular focus on their relationship to inversions.

\begin{proposition}\label{cdis}
Let $(P, \leq, \odot, \rightarrow, \leftarrow)$ be a Girard poset equipped with an inversion operation ${}^{\xprime}$. Then $P$ is a residuated poset with inversion. Moreover, the cyclic dualizing element $d \in P$ is uniquely given by
\[
d = \bigvee_{x \in P} (x \odot x^{\xprime}).
\]
\end{proposition}

\begin{proof}Since ${}^{\xprime}$ is an inversion on $P$, for any $x \in P$ we have
\[
x^{\xprime} = \bigvee \{ t \in P \mid t \odot x \leq d \} = \bigvee \{ t \in P \mid x \odot t \leq d \}.
\]
Hence,
\[
x \odot x^{\xprime} = \bigvee \{ x \odot t \mid x \odot t \leq d \} \leq d,
\]
and similarly,
\[
x^{\xprime} \odot x \leq d.
\]
Now, if $z \in P$ satisfies $z \geq x \odot x^{\xprime}$ for every $x \in P$, then in particular,
\[
z \geq d \odot d^{\xprime} = d,
\]
since $d^{\xprime}$ is the unit in $P$. Combining these, we conclude that
\[
d = \bigvee_{x \in P} (x \odot x^{\xprime})
\]
holds uniquely.
\end{proof}

The preceding discussion naturally leads to a stronger structural result linking Girard posets to Boolean algebras.

\begin{proposition}\label{charGirBooldual} Let $(P, \leq, \odot, \rightarrow, \leftarrow)$ be a Girard poset equipped with an inversion operation ${}^{\xprime}$.
Then $P$ is a Boolean algebra if and only if $P$ is an idempotent residuated lattice with a cyclic dualizing element $d = 0$.
\end{proposition}
\begin{proof}($\Rightarrow$) From Proposition \ref{cdis}, we have for every $x \in P$ that
\[
x \odot x^{\xprime} = x \wedge x^{\xprime} = 0,
\]
which implies that
\[
d = 0.
\]

\noindent{}($\Leftarrow$) For any $x, y \in P$, since $d^{\xprime} = 1$ is the unit in $P$, it follows that
\[
x \odot y \leq x \odot 1 = x \quad \text{and} \quad x \odot y \leq 1 \odot y = y.
\]
Moreover,
\[
x \wedge y = (x \wedge y) \odot (x \wedge y) \leq x \odot y.
\]
Therefore, we conclude that
\[
x \odot y = x \wedge y,
\]
which shows that $P$ is a distributive lattice. Furthermore, the elements $x$ and $x^{\xprime}$ are complementary since 
$d = 0$ implies 
\[
x \odot x^{\xprime} = x \wedge x^{\xprime} = 0.
\]
\end{proof}

The following theorem establishes the precise algebraic conditions under which a unital residuated poset is also a Girard poset, providing necessary and sufficient criteria based on the existence of an inversion operation.

\begin{theorem}\label{charth} Let $(P, \leq, \odot, \rightarrow, \leftarrow)$ be a unital residuated poset  with a unit $e$. The following statements  are equivalent:
\begin{enumerate}
    \item $P$ is a Girard poset.
    \item $P$ is a residuated poset with an inversion $^{\xprime}$ and $x^{\xprime} = x \rightarrow e^{\xprime} = e^{\xprime} \leftarrow x$ for every $x \in P$.
    \item $P$ is a residuated poset with an inversion $^{\xprime}$ and for every $x, y, t \in P$ it holds $t \odot x \le y^{\xprime} \Leftrightarrow y \odot t \le x^{\xprime}$.
\end{enumerate}
\end{theorem}

\begin{proof}
$1 \Rightarrow 3$: If $P$ is a Girard poset then the linear negation $^{\xprime}$ on $P$ is an inversion and 
$x \rightarrow y^{\xprime} =  x^{\xprime}\leftarrow y$ for all $x, y \in P$ (see the remark after Definition \ref{def:Girpos}). Therefore, for any $t, x, y \in P$, 
$$t \odot x \le y^{\xprime} \Leftrightarrow t \le x \rightarrow y^{\xprime} = x^{\xprime}\leftarrow y \Leftrightarrow y \odot t \le x^{\xprime}.$$

\noindent$3 \Rightarrow 2$: From the condition, taking $y = e$, we have  $t \odot x \leq e^{\xprime} \iff t \leq x^{\xprime}$. This implies $x \rightarrow e^{\xprime} \leq x^{\xprime}$. 

Similarly, using $x \odot t \leq e^{\xprime} \iff t \leq x^{\xprime}$, we obtain $ e^{\xprime} \leftarrow x \leq x^{\xprime}$. 

Moreover, since $x^{\xprime} \odot x \leq e^{\xprime}$ and $x \odot x^{\xprime} \leq e^{\xprime}$, it follows that $x^{\xprime} \leq x \rightarrow e^{\xprime}$ {and} $x^{\xprime} \leq e^{\xprime} \leftarrow x$. Hence, $$ x^{\xprime} = x \rightarrow e^{\xprime} = e^{\xprime} \leftarrow x. $$

\noindent$2 \Rightarrow 1$: Given that $x^{\xprime} = x \rightarrow e^{\xprime} = e^{\xprime} \leftarrow x$ and  that $x^{\xxprime} = x$ holds we conclude that $e^{\xprime}$ is a cyclic dualizing element and $P$ is a Girard poset with the linear negation $^{\xprime}$.
\end{proof}

\section{Quantales carrying ortholattice structure}
\label{quancarr}

Our primary motivation for this paper was to study a lattice structure—specifically, a residuated orthomodular lattice—that sits at the intersection of the many-valued logics and quantum logics frameworks.

We note that related research has already explored an ortholattice satisfying a weaker condition of adjointness, as documented in \cite{Fussner_2021}.

\begin{proposition}\label{compres1}
    Let $(P, \le, \odot, \rightarrow, \leftarrow)$ 
    be an integral residuated lattice that is complemented. 
    Then 
    \begin{enumerate}
\item $P$ is idempotent.
\item $\odot=\wedge$.
\item $P$ is a Boolean algebra.
\end{enumerate}
\end{proposition}
\begin{proof} 1: Let $x\in P$. Since 
$P$ is integral we conclude
$x\odot x'\leq x\wedge x'=0$. We obtain 
$$
x=x\odot 1= x\odot (x\vee x')=x\odot x\vee x\odot x'=x\odot x.
$$
2: Let $x, y\in P$. We compute using integrality and idempotency of $P$:
$$
x\odot y\leq x\wedge y=(x\wedge y)\odot (x\wedge y)%
\leq x\odot y.
$$
3: Since $\odot=\wedge$ and $\odot$ distributes over 
finite joins we conclude that $P$ is a distributive lattice. 
Hence by Remark \ref{boolcompor} we obtain that 
$P$ is a Boolean algebra. 
\end{proof}

We immediately obtain the following theorem. 

\begin{theorem}
\label{thm:2}
A complemented lattice is an integral residuated lattice iff it is a Boolean algebra.
\end{theorem}

\begin{remark}
    Theorem \ref{thm:2} has a precursor in Theorem 7.31 of \cite{WaDi39}, which states that a complemented lattice is 
    a commutative integral residuated lattice if and only if it is a Boolean algebra. Note that our proof proceeds independently. A related theorem for orthomodular lattices is also Theorem 2 in 
    \cite{TkTu11}.
\end{remark}

\begin{corollary}\label{intquaBool}
    A complemented lattice is an integral quantale iff it is a complete Boolean algebra.
\end{corollary}

\begin{proposition}\label{compres}
    Let $(P, \le, \odot, \rightarrow, \leftarrow)$ 
    be an unital residuated lattice with unit $e$ that is orthomodular with orthocomplementation ${}^{\perp}$. 
    Then 
    \begin{enumerate}
\item $\downset e$ is a Boolean algebra with orthocomplementation ${}^{{\perp}_{e}}$ such that every two elements $x, y\in \downset e$ are compatible in $P$. 
\item $\downset e \cup \upset e^{\perp}$ is contained in a block $B$ of $P$.
\end{enumerate}
\end{proposition}
\begin{proof} 1: Since $e\odot e=e=e\wedge e$ we conclude that 
$\downset e$ is an integral residuated lattice such that  multiplication, meet and join in $\downset e$ are restrictions 
from multiplication, meet  and join in $P$. 
From Lemma \ref{DownsetLem} we know that $\downset e$ 
is an orthomodular lattice. Hence $\downset e$ is a Boolean algebra with orthocomplementation ${}^{{\perp}_{e}}$. 

Assume that $x, y\in \downset e$. Then $x$ and $y$ are 
    compatible in $\downset e$. Hence 
    \begin{align*}
        x &= (x \wedge y) \vee (x \wedge y^{{\perp}_{e}})=
(x \wedge y) \vee (x \wedge y^\perp\wedge e)\\
&=%
(x \wedge y) \vee (x \wedge y^\perp).
    \end{align*}

\noindent{}2: Since every two elements $x, y\in \downset e$ are compatible in $P$ the set $\downset e$ is contained in some block $B$ of $P$. Since $B$ is closed under operation 
${}^{\perp}$ we conclude that  $\downset e \cup \upset e^{\perp}\subseteq B$.
\end{proof}

{Corollary \ref{intquaBool}} prompts a natural question: Are there examples of orthomodular unital quantales that are \textbf{not} Boolean algebras?

This question has a positive answer, which we establish using {Example \ref{examHilb}} and the result from {\cite[Corollary 13]{KrEg}}. Specifically, the complete lattice $C(M_{n}(\mathbb{C}))$, which consists of the closed subspaces of the Hilbert space $M_{n}(\mathbb{C})$ (for $n \geq 2$), serves as a counterexample. This lattice is simultaneously an {orthomodular lattice} and a {non-commutative involutive Girard quantale}.

This finding naturally extends to a further inquiry: Are there examples of {orthomodular, unital, and commutative quantales} that are not Boolean algebras?

\begin{theorem}\label{Rnexample}
    Let $n\in \Naturals$. Then $C(\Reals^{n})$ is an  
    {orthomodular lattice} and a {commutative Girard quantale} 
    such that ${}^{\perp}={}^{\xprime}$.
\end{theorem}
\begin{proof} From {Example \ref{examHilb}}, we conclude that 
    $C(\Reals^{n})$ is an  {orthomodular lattice} with inversion  ${}^{\perp}$. Recall that 
    $$
    \begin{array}{r@{\,\,}l}
U^{\perp}=\{(a_1, \dots, a_n)\in \mathbb{R}^{n}\mid& \sum_{i=1}^{n} a_i\cdot u_i=0\\ 
&\text{for every}\ 
(u_1, \dots, u_n)\in U\}.
\end{array}
    $$

Let us prove that $C(\mathbb{R}^{n})$ is a {commutative Girard quantale}.
Recall that $\mathbb{R}^{n}$ is a ring with addition and multiplication defined {pointwise}. For every $s\in \mathbb{R}^{n}$ we assume that $s=(s_1, \dots, s_n)$.

Define the multiplication \( \odot \) of subspaces \( S, T \subseteq \mathbb{R}^{n} \) as the linear span of products:
\[
S \odot T = \lin{\{ s\cdot t \mid s \in S, t \in T \}}
\]

\begin{itemize}
  \item This is well-defined, as the span of products of elements from subspaces is again a subspace.
  \item It is associative and commutative:
  \begin{align*}
   (S \odot T) \odot U &= S \odot (T \odot U),\\
    S \odot T &= T \odot S,
  \end{align*}
 due to associativity and commutativity of ring multiplication.
  \item It distributes over joins, i.e., for family \( \{S_i\} \) of subspaces:
  \[
  \left( \bigvee_i S_i \right) \odot T = \bigvee_i (S_i \odot T),
  \]
  similarly for the right side.
  \item Select the unit $e$ as a subspace 
  $\lin{\{(1, \dots, 1)\}}$. Evidently, 
  $$
e\odot S=S=S\odot e
  $$
  for every subspace  \( S \subseteq \mathbb{R}^{n} \).
  \item Select the dualizing element $d$ as a subspace 
  $$e^{\perp}=%\lin{\{(1, \dots, 1)\}}^{\perp}=%
  \{(a_1, \dots, a_n)\in \mathbb{R}^{n}\mid \sum_{i=1}^{n} a_i=0\}.$$
\end{itemize}

Let \( S, T, U \subseteq \mathbb{R}^{n} \) are subspaces of 
$\mathbb{R}^{n}$. We compute:
\begin{align*}
    S\odot  T\subseteq U^{\perp} &\Leftrightarrow 
    \sum_{i}^{n} s_i\cdot t_i\cdot u_i=0\ \text{for all}\ 
    s\in S,t\in  T, u\in U\\
    &\Leftrightarrow 
    \sum_{i}^{n} u_i\cdot t_i\cdot s_i=0\ \text{for all}\ 
    s\in S,t\in  T, u\in U\\
    &\Leftrightarrow U\odot  T\subseteq S^{\perp}.
\end{align*}

From Theorem \ref{charth} we conclude that these properties make \( C(\mathbb{R}^{n}) \) a 
commutative Girard quantale.
\end{proof}

In what follows we present 
natural examples of orthomodular lattices based on infinite-dimensional Hilbert spaces that are 
non-commutative involutive Girard quantales.

\begin{example}
The spectrum of a $C^*$-algebra $A$, denoted $\mathrm{Max}\, A$, is the sup-lattice of all linear subspaces of $A$ which are closed with respect to the norm topology. It is an involutive unital quantale with respect to the multiplication and involution 
\[
a\cdot b = \overline{\{A\circ B \mid A \in a,\, B \in b\}}%
\ \text{and}\ a^{*}={\{A^{*} \mid A \in a,\}},
\]
where the overline denotes closure.
    The spectrum quantale $\mathrm{Max}\, \mathcal{B}(H)$ 
    of the $C^*$-algebra $\mathcal{B}(H)$ of bounded linear 
    operators on an infinite-dimensional Hilbert space $H$ is no longer Girard.

    The Frobenius scalar product ($\langle A\mid B\rangle = \mathrm{Tr}(A^*\circ B)$ where $\mathrm{Tr}$ means the trace) still establishes a vital sense of orthogonality—a property that holds true for individual operators as well as for the corresponding subspaces (see \cite{KrEg}).

Operator $A$ is said to be \emph{of finite (Schatten) $p$-norm} if $\mathrm{Tr}\,|A|^p<\infty$.
In particular, $A$ is \emph{trace class} if it is of finite $1$-norm, or \emph{Hilbert--Schmidt} if it is of finite $2$-norm.
An ordinary bounded operator can be regarded as having finite $\infty$-norm.
Given two operators $A,B$, the former of finite $p$-norm and the latter of finite $q$-norm, their product is of finite $r$-norm where
\begin{equation}\label{pq}
\frac{1}{r}=\min\left\{ 1,\frac{1}{p}+\frac{1}{q}\right\}.
\end{equation}
Thus the operators $A,B$ could be considered \emph{orthogonal} whenever $p^{-1}+q^{-1}\geq 1$ and $\mathrm{Tr}(A^*\circ B)=0$.
Notice that the antitone mapping $f\colon [1,\infty]\to [0,1]$, $f(p)=1-\frac{1}{p}$,
together with \eqref{pq} yields the \emph{\L ukasiewicz product} $\odot_{\text{\emph{\L}}}$ on $[0,1]$, namely
\[
a\odot_{\text{\emph{\L}}} b=\max\{ 0,a+b-1\}.
\]
The orthogonality relation can be naturally extended to \emph{closed subspaces of Schatten classes}.

More precisely, for every $p\in [1,\infty]$ we consider a \emph{$p$-norm spectrum} $\mathrm{Max}_p\mathcal{C}_p(H)$ as a quantale of all $p$-norm closed subspaces of the Schatten class $\mathcal{C}_p(H)$ (consisting of finite $p$-norm operators).
If $p<q$, then every subspace of $\mathcal{C}_p(H)$ can be embedded into $\mathcal{C}_q(H)$ and then closed in the $q$-norm, and this yields a quantale homomorphism $\mathrm{Max}_p\mathcal{C}_p(H)\to\mathrm{Max}_q\mathcal{C}_q(H)$.

For $a\in\mathrm{Max}_p\mathcal{C}_p(H)$ and $b\in\mathrm{Max}_q\mathcal{C}_q(H)$, we find that $a\cdot b$ is in $\mathrm{Max}_r\mathcal{C}_r(H)$ for $r$ satisfying \eqref{pq}.
Thus the interval scheme of spectra
\[
\mathrm{Max}_1\mathcal{C}_1(H)\to\cdots\to\mathrm{Max}_p\mathcal{C}_p(H)\to\cdots\to\mathrm{Max}_\infty\mathcal{C}_\infty(H),
\]
where $\mathrm{Max}_\infty\mathcal{C}_\infty(H)=\mathrm{Max}\,\mathcal{B}(H)$, 
is organized in the same way as the \L ukasiewicz logic on $[0,1]$.

The dualizing element here is the subspace of $\mathcal{C}_1(H)$ consisting of \emph{trace-zero} operators
\[
d=\{ A\mid \mathrm{Tr}\,|A|<\infty,\, \mathrm{Tr}\,A=0\}.
\]

A complement of $a\in\mathrm{Max}_p\mathcal{C}_p(H)$ is in
$\mathrm{Max}_{p'}\mathcal{C}_{p'}(H)$ for $p'=(1-p^{-1})^{-1}$ and is given
by
$$a^\perp=\{ B\mid \text{Tr}\,(A^*\circ B)=0\text{ for all }A\in a\}.$$

The scheme can be made into a single involutive Girard quantale $G$ consisting of
mappings assigning to each $p\in [1,\infty]$ an element $a_p$ of
$\mathrm{Max}_p\mathcal{C}_p(H)$ provided that $a_p\subseteq a_q$ whenever
$p\leq q$.
The join and involution are given componentwise, the product by rule
$$(a\odot b)_r=\bigvee_{1/r\leq 1/p+1/q}a_p \cdot b_q,$$
and the orthocomplementation by rule
$(a^\perp)_p=(a_{p'})^\perp$.

Finally, let $a\in\mathrm{Max}_p\mathcal{C}_p(H)$ and assume that the norm
closed subspace $a\oplus a^\perp$ is a proper subspace of
$\mathcal{B}(H)$.
Then there is a non-zero (trace class) complement of $a\oplus a^\perp$
and its elements are perpendicular both to $a$ and $a^\perp$.
But an operator $A$ from it would satisfy $\text{Tr}\,(A^*\circ A)=0$ which is
a contradiction.
Hence, $a\oplus a^\perp=H$ and every operator in $\mathcal{B}(H)$
(uniquely) decomposes to $a$ and $a^\perp$ like in finitely dimensional
spaces.
If
$a_p\in\mathrm{Max}_p\mathcal{C}_p(H),a_q\in\mathrm{Max}_q\mathcal{C}_p(H)$
are such that $p\leq q,a_p\subset a_q$, then there is an operator $A\in
a_q\smallsetminus a_p$ with a non-zero projection on $a_p^\perp$.
We infer that $a_p^\perp\cap a_q\neq 0$ and $\perp$ provides an
orthomodular lattice structure on $G$.

\end{example}

\section{Conclusions}\label{Conclusions}

This paper has explored the fertile intersection between two major algebraic traditions in non-classical logic: residuated structures from many-valued logic and orthomodular lattices from quantum logic. Our investigation reveals both limitations and possibilities for unifying these frameworks.
On the restrictive side, we established that complemented lattices admit an integral residuated structure if and only if they are Boolean algebras (Theorem~\ref{thm:2}).

This result, which strengthens previous work by removing the commutativity assumption, shows that Boolean algebras represent the unique meeting point of complementation and integral residuation. This finding delineates a clear boundary: any attempt to construct richer non-Boolean structures must venture beyond the integral case.

Our main constructive contribution demonstrates that this boundary can indeed be crossed in the unital setting. Through Theorem~\ref{Rnexample}, we showed that the lattice $C(\mathbb{R}^n)$ 
of closed subspaces of real coordinate space simultaneously carries the structure of an orthomodular lattice and a commutative Girard quantale, with the orthocomplement coinciding with the linear negation. This provides a concrete family of examples where quantum-logical and many-valued logical structures coexist harmoniously in a non-Boolean framework.

The characterization theorem (Theorem~\ref{charth}) establishes a fundamental connection between Girard posets and residuated posets with inversions, offering three equivalent conditions that illuminate the relationship between linear negation, cyclic dualizing elements, and adjointness properties. This characterization provides a powerful tool for recognizing and constructing Girard posets in various mathematical contexts.

\subsection*{Future Directions}
Several natural questions emerge from this work:

\begin{enumerate}
\item \textbf{Complex coordinate spaces:} While we established that $C(\mathbb{R}^n)$
is both orthomodular and a commutative Girard quantale, the situation for $C(\mathbb{C}^n)$
 remains open. Does the lattice of closed subspaces of complex coordinate space admit a compatible quantale structure? If not, what structural obstacles prevent such a construction?

\item \textbf{Logical completeness:} What logic corresponds to orthomodular Girard quantales? Can we develop a sound and complete proof system that captures reasoning in these structures, potentially combining aspects of quantum logic and linear logic?

\end{enumerate}

The framework developed in this paper opens new avenues for algebraic investigations at the boundary of quantum and many-valued logics. By identifying both impossibility results and constructive examples, we have mapped some of the terrain where these logical traditions intersect. The explicit construction of orthomodular Girard quantales demonstrates that non-Boolean structures can indeed accommodate both quantum complementation and residuation principles, suggesting rich possibilities for further exploration in algebra, logic, and their applications to quantum theory.

\section*{Acknowledgment}

The first author acknowledges the support of project 23-09731L by the Czech Science Foundation (GA\v CR), entitled 
``Representations of algebraic semantics for substructural logics''.
The research of the second author and third author 
was funded in part by the Austrian Science Fund (FWF) 10.55776/PIN5424624 and the Czech Science Foundation
(GACR) 25-20013L.

\end{document}